\documentclass[12pt]{article}
\usepackage{makeidx}     
\usepackage{enumerate} 
\usepackage{graphicx,epsfig,lscape,float}    
\usepackage{amsmath,array,amssymb, amsfonts, latexsym, mathrsfs, dsfont,amsthm}
\usepackage{algorithmic, algorithm, caption}

\makeatletter
\DeclareCaptionLabelFormat{numberless}{\ALG@name#1}
\makeatother

\usepackage{varwidth}

\usepackage{psfrag,fancybox}

\usepackage{tikz}
\usetikzlibrary{shapes,decorations}
\usetikzlibrary{arrows}
\usepackage{rotating}
\usepackage{pgfplots}
\usepgfplotslibrary{units}

\usetikzlibrary{calc}

\usepackage{multicol}    

\makeindex             
\newcommand{\sign}{\mathbf{sign}}

\newcommand{\be}{\begin{equation}}
\newcommand{\ee}{\end{equation}}

\newcommand{\restr}[1]{\lower0.4ex\hbox{$\vert$}\lower0.7ex\hbox{ $\!_{#1}$ }}

\newtheorem{theorem}{Theorem}[section]
\newtheorem{proposition}{Proposition}[section]

\newtheorem{corollary}[proposition]{Corollary}
\newtheorem{lemma}[proposition]{Lemma}

\newtheorem{remark}[proposition]{Remark}

\usepackage{amsmath,array,amssymb,amsfonts,graphicx}

 \title{Direct solution of piecewise linear systems}

 \author{ Manuel Radons}

\begin{document}
\maketitle

\textbf{Abstract:}
{\em
Let $S$ be a real $n\times n$ matrix, $z,\hat c\in \mathbb R^n$, and $\vert z\vert$ the componentwise modulus of $z$. 
Then the piecewise linear equation system $$z-S\vert z\vert = \hat c$$ is called an \textit{absolute value equation} (AVE). 
It has been proven to be equivalent to the general \textit{linear complementarity problem}, which means that it is NP hard in general.

We will show that for several system classes the AVE essentially retains the good natured solvability properties of regular linear systems. I.e., it can be solved directly by a slightly modified Gaussian elimination that we call the signed Gaussian elimination. For dense matrices $S$ this algorithm has the same operations count as the classical Gaussian elimination with symmetric pivoting. For tridiagonal systems in $n$ variables its computational cost is roughly that of sorting $n$ floating point numbers. The sharpness of the proposed restrictions on $S$ will be established.

   \em} 
   
   \textbf{Keywords} 
Absolute value equation; Linear complementarity problem; Piecewise linear equation system; Direct solver; Signed Gaussian elimination  \\

\textbf{MSC 2010} $15\mathrm A39,$ \quad  $65\mathrm K05,$ \quad  $90\mathrm C33$

\section{Introduction and notation}

We denote by $\operatorname{M}_n(\mathbb R)$ the space of $n\times n$ real matrices, and by $[n]$ the set $\{1,\dots,n\}$. 
For vectors and matrices \textit{absolute values and comparisons are used entrywise}. Zero vectors and matrices are denoted by $\mathbf 0$.

A \textit{signature matrix} $\Sigma$, or, briefly, \textit{a signature}, is a diagonal matrix with entries $+1$ or $-1$. The set of $n$-dimensional signature matrices is denoted by $\operatorname{diag}_{n,\sigma}$. A single diagonal entry of a signature is a sign $\sigma_i$ ($i\in [n]$).

Let $S\in \operatorname{M}_n(\mathbb R)$, $z,\hat c\in \mathbb R^n$. The \textit{piecewise linear equation system} (PLE)

\begin{align}\label{AVE_std}
 z - S\vert z\vert = \hat c 
\end{align}
is called an \textit{absolute value equation} (AVE). It was first introduced by Rohn in \cite{rohn1989interval}. Mangasarian proved its equivalence to the general \textit{linear complementarity problem} (LCP) \cite{mangasarian2006absval}. In \cite[pp. 216-230]{neumaier1990interval} Neumaier authored a detailed survey about its intimate connection to the research field of \textit{linear interval equations}. A recent result by Griewank and Streubel  has shown that PLEs of arbitrary structure can be, with a \textit{one-to-one solution correspondence}, transformed into an AVE \cite[Lem. 6.5]{griewank2014abs}. 

An especially closely related system type are equilibrium problems of the form 
\begin{align}\label{brugEq}
Ax+ \max (0, x)= b,
\end{align}
where $A\in \operatorname{M}_n(\mathbb R)$ and $x, b \in \mathbb R^n$. (A prominent example is the first hydrodynamic model presented in \cite{brugnano2008iterative}.) Using the identity $\max(s, t) = (s+t +\vert s-t \vert)/ 2$, equality \eqref{brugEq} can be reformulated as
\begin{align}\label{brugAVE}
Ax + \frac{x+\vert x\vert}{2}=b\quad \Longleftrightarrow \quad (2A+I)x +\vert x \vert \equiv Bx+\vert x\vert=2b.
\end{align}
For regular $B$, system \eqref{brugAVE} is clearly equivalent to \eqref{AVE_std}.

This position at the crossroads of several interesting problem areas gives relevance to the task of developing efficient solvers for the AVE. The latest publications on the matter include approaches by linear programming \cite{mangasarian2014abs} and concave minimization \cite{mangasarian2007concave}, as well as a variety of Newton and fixed point methods (see, e.g., \cite{brugnano2008iterative}, \cite{yuan2012iterative}, \cite{hu2011absval} or \cite{griewank2014abs}).  

Let $\Sigma \in \operatorname{diag}_{n,\sigma}$ s.t. $\Sigma z=\vert z\vert$. (Note that, since $0=+0=-0$, we need no "$0$"-sign.) Then we can rewrite \eqref{AVE_std} as 
\begin{align}\label{AVE}
  (I-S\Sigma)z = \hat c.
\end{align}
In this form it becomes apparent that the main difficulty in the computation of a solution for \eqref{AVE} is to determine the proper signature $\Sigma$ for $z$. That is, to determine in which of the $2^n$ orthants about the origin $z$ lies. This is NP-hard in general \cite{mangasarian2007absval}.

It was proven by Rump in \cite[Cor. 2.9]{rump1997theorems} that checking the system for unique solvability is NP-hard as well, 
as it is equivalent to checking whether a quantity called the \textit{sign-real spectral radius} of $S$ is smaller than one, which in turn is equivalent to checking whether the system matrix of the equivalent LCP is a \textit{$P$-matrix}. 
As these notions and results are fundamental to the understanding of the AVE, we will give a short account of them in the second section. There we will also see that the systems investigated in the present paper, for all of which it holds $\lVert S\lVert_{\infty} <1$, are uniquely solvable.  

The following simple observation is key to the subsequent discussion: 
\begin{proposition} Let $S\in \operatorname M_n(\mathbb R)$ and $z,\hat c\in\mathbb R^n$ such that they satisfy \eqref{AVE}. Then, if $\lVert S\lVert_{\infty}<1$, for at least one $i\in [n]$ the signs of $z_i$ and $\hat c_i$ have to coincide. 
\end{proposition}
\begin{proof}
Let $z_i$ be an entry of $z$ s.t. $\vert z_i\vert \ge \vert z_j\vert\ $ for all $j\in [n]$. If $z_i=0$, then $z=\mathbf 0$ and thus $\hat c\equiv z-S\vert z\vert$ is the zero vector as well -- and the statement holds trivially. If $\vert z_i\vert >0$, then $\left| e_i^TS\vert z\vert \right| < \vert z_i\vert$, due to the norm constraint on $S$. Thus, $\hat c_i = z_i - e_i^TS\vert z\vert$ will adopt the sign of $z_i$. 
\end{proof}

We do not know though, for which indices the signs coincide. In the third section we will derive several types of structural restrictions on $S$, each of which will guarantee the coincidence of the signs of $z_i$ and $\hat c_i$ for all $i\in [n]$ with $\vert \hat c_i\vert $ maximal in $\hat c$.  
   
In the fourth paragraph we will devise a modified Gaussian elimination that exploits this knowledge. This \textit{signed Gaussian elimination} (SGE) will base on the following central points:  
\begin{itemize}
\item We are enabled to perform one step of Gaussian elimination on the AVE in the form \eqref{AVE}, if we know the correct sign of $z_1$.
\item If $\lVert S\lVert_{\infty}<1$, no row or column pivot causes numerical instabilities in the performance of a Gaussian elimination step on system \eqref{AVE}. Hence, we can always produce a constellation, where $\vert \hat c_1\vert$ is maximal in $\hat c$.
\item The restrictions on $S$ developed in the third paragraph are invariant under Gaussian elimination steps.
\end{itemize}
For $S$ that conform to the restrictions derived in paragraph three, the first two points mean that we can always perform one Gaussian elimination step on system \eqref{AVE}. The third point ensures that we can repeat the procedure for the reduced system(s) and ultimately calculate the correct (unique) solution of the AVE.  

We will briefly analyze the modified algorithm's runtime in the dense and tridiagonal case. 
For a dense matrix $S$ the SGE has the operations count of a Gaussian elimination with symmetric pivoting. 
For the tridiagonal SGE the supplementary operations cost roughly as much as sorting $\hat c$ with respect to the absolute value of its entries. As the underlying tridiagonal Gaussian elimination, also known as the Thomas Algorithm, is in $\mathcal O(n)$, this means that the asymptotical complexity of the modified algorithm depends, at the current state of research, one-to-one on the implementation of the extra effort.  

The paper is concluded by a discussion of the sharpness of the proposed restrictions on $S$.  

For readers primarily interested in the algorithmic results, we remark that inequality \eqref{inftyBound}, equivalence $1.\Leftrightarrow 3.$ from Theorem \ref{unique}., and the statements of Theorem \ref{mainThm}. present the most basic preknowledge that should enable them to work with the fourth paragraph. 

Note that we already outlined the approach described above in \cite[Parag. 7]{griewank2014abs}. This paper presents the announced elaboration on the concept.

\section{Sign-real spectral radius}

Denote by $\rho (S)$ the spectral radius of $S$ and let 
\begin{align*}
\rho_0(S)\equiv \max \{\vert \lambda\vert: \lambda\ \text{\textit{real}\ eigenvalue of $S$}\}
\end{align*}
be the \textit{real spectral radius} of $S$. Then its sign-real spectral radius is defined as follows (see \cite[Def. 1.1]{rump1997theorems}): 
\begin{align*}
   \rho_0^s(S) \; \equiv \; \max \left \{ \rho_0(\Sigma S ) : \Sigma \in \operatorname{diag}_{n,\sigma}  \right \}. 
\end{align*}
The exponential number of signatures $\Sigma$ accounts for the NP-hardness of the computation of $\rho_0^s(S)$. 
It is easy to check that $\operatorname{diag}_{n,\sigma}$ is a finite subgroup of $\operatorname{Gl}_n(\mathbb R)$. Thus, for a fixed signature $\bar \Sigma$, the sets $\{\Sigma(\bar\Sigma S): \Sigma \in \operatorname{diag}_{n,\sigma}\}$ and $\{\Sigma S: \Sigma \in \operatorname{diag}_{n,\sigma}\}$ are identical modulo a permutation. Furthermore, since all $\Sigma \in  \operatorname{diag}_{n,\sigma}$ are obviously involutive, i.e., $\Sigma^{-1}=\Sigma$, the spectra of $S$ and $\Sigma S\Sigma$ are identical. These observations immediately yield the useful identity
\begin{align*}
\rho_0^s(S)=\rho_0^s(\Sigma_1 S)=\rho_0^s(S\Sigma_2)=\rho_0^s(\Sigma_1S\Sigma_2)\qquad \forall\ \Sigma_1, \Sigma_2 \in \operatorname{diag}_{n,\sigma}.
\end{align*} 
Recall that a real (or complex) square matrix is called a $P$-matrix if every principal minor is positive \cite[p. 147]{cottle1992lcp}. An LCP has a unique solution for all right hand sides if and only if its system matrix is a $P$-matrix \cite[p. 148, Thm. 3.3.7]{cottle1992lcp}. 
We will now (re-) prove some essential facts about the relation between $\rho_0^s(S)$ and the solvability properties of \eqref{AVE}.

\begin{theorem}\label{unique}
 Let $S\in\operatorname M_n(\mathbb R)$. Then the following are equivalent:
\begin{enumerate}
 \item $ \rho_0^s(S)\ <\ 1$.
 \item $(I-S)^{-1}(I+S)$ is a $P$-matrix.
 \item The system $(I-S\Sigma)z = \hat c$ has a unique solution for all $\hat c\in\mathbb R^n$.
 \item The function $\varphi:\ \mathbb R^n \rightarrow \mathbb R^n,\ z\rightarrow z+S\vert z\vert$ is bijective.
 \item For all $\hat c\in\mathbb R^n$ there exists a unique $\Sigma \in \operatorname{diag}_{n,\sigma}$ s.t. $b\equiv (I-S\Sigma)^{-1}\hat c$ lies in the orthant defined by $\Sigma$.
 \item $\operatorname{det}(I-S\Sigma)\ >\ 0$ for all $\Sigma\in \operatorname{diag}_{n,\sigma}$.
 \item $\operatorname{det}(I-SD)\ >\ 0$ for all real diagonal matrices $D\in\operatorname M_n(\mathbb R)$ with $\lVert D\lVert_{\infty}\le 1$.
 \end{enumerate}
\end{theorem}

\begin{proof}
Let $\mu, \lambda \in \mathbb R$. We note that, since $\operatorname{det}((A+\mu I) - \lambda I) = \operatorname{det}(A -(\lambda -\mu)I)$, the spectrum of a matrix $\bar A\equiv A+\mu I$ is the spectrum of $A$ shifted by $\mu$ along the real axis. We will refer to this fact by the abbreviation SHIFT.

$1.\Rightarrow 6.:$ Let $ \rho_0^s(S)<1$ and fix a signature $\Sigma$. Then the absolute value of all real eigenvalues of $S\Sigma$ is smaller than one. Thus, by SHIFT, all real eigenvalues of $(I-S\Sigma)$ lie in the open interval $(0,2)$, which means that their product is positive. The complex eigenvalues appear in conjugate pairs, hence their product is positive as well. This yields a positive determinant. 

$6.\Rightarrow 7.:$ Let $\Sigma, \Sigma'\in \operatorname{diag}_{n,\sigma}$ be signatures that differ only in the first sign. By assumption we have 
$$\operatorname{det}(I-S\Sigma)\ >\ 0\ \text{and}\ \operatorname{det}(I-S\Sigma')\ >\ 0\ . $$ Then, by the linearity of the determinant for rank-$1$ updates, it holds $$\operatorname{det}(I-SD)\ >\ 0\ ,$$ where $D\in\operatorname M_n(\mathbb R)$ is a diagonal matrix whose first entry lies in the interval $[-1,1]$, while all others equal the corresponding entries in $\Sigma\  /\ \Sigma'$. Now apply this argument inductively.   

$7.\Rightarrow 1.:$ Assume that $\operatorname{det}(I-S\Sigma)>0$ for all for all real diagonal matrices $D\in\operatorname M_n(\mathbb R)$ with $\lVert D\lVert_{\infty}\le 1$, but $\rho_0^s(S)\ge 1$. Then there exists a signature $\Sigma\in \operatorname{diag}_{n,\sigma}$ s.t. $S\Sigma$ has at least one real eigenvalue $\lambda$ with $\vert \lambda\vert \ge 1$. Define $D:= \frac 1\lambda \Sigma$. Clearly, $D$ is a diagonal matrix with $\Vert D\Vert_\infty \le 1$. And, by SHIFT, it holds $$\operatorname{det}(I-SD)\ =\ 0$$ -- in contradiction to the hypothesis.

$2.\Leftrightarrow 3.:$
 Let $z\equiv u-w$ with $u \perp w$ in that $u \geq 0 \leq w$ and $u^\top w =0$, we obtain $|z| = u+w$. Substituting this into the AVE, we get
 \begin{align*}
   \qquad \hat c\; &=\; u-w + S (u+w)    \\
  \Longleftrightarrow  \qquad (I-S)w\; &=\; -\hat c + (I+S)u \\
  \Longleftrightarrow  \; \qquad \quad \quad \quad w\; &=\; -(I-S)^{-1} \hat c + (I-S)^{-1}(I+S)u. \\
 \end{align*}
The latter equation has the form of an LCP and hence possesses a unique solution if and only if $(I-S)^{-1}(I+S)$ is a $P$-matrix.

$3.\Leftrightarrow 5.:$ If $b$ lies in the orthant defined by $\Sigma$, then $\Sigma b= \vert b\vert$, that is, $b$ is a solution of the system. But then the equivalence is clear.

$2.\Leftrightarrow 7.:$ For $A,B\in \operatorname{M}_n(\mathbb R)$ the following equivalency holds:  $TA+(I-T)B$ is regular for all $n$-dimensional diagonal matrices $T$ with entries $t_i\in [0,1] \Longleftrightarrow A^{-1}B$ is a $P$-matrix \cite[Thm. 3.4]{johnson1995pmatrix}.

But we have $T(I-S)+(I-T)(I+S)=I-(I-2T)S$ -- and the set of matrices $I-2T$ is clearly identical to the set of diagonal matrices $D$ with $\lVert D\lVert_{\infty}\le 1$.

$7.\Rightarrow 6.:$ Obvious, since the $n$-dimensional diagonal matrices $D$ with $\lVert D\lVert_{\infty} \le 1$ are the convex hull of $\operatorname{diag}_{n,\sigma}$. 

$6.\Rightarrow 4.:$ If we interpret \eqref{AVE_std} as the piecewise linear function $$\varphi: \mathbb R^n\ \rightarrow\ \mathbb R^n,\ z\ \rightarrow\ z+S\vert z\vert\ ,$$ then $\operatorname{det}(I-S\Sigma)>0$ for all signatures means that the limiting Jacobians of $\varphi$ all have the same determinant sign -- a property which is called coherent orientation and implies surjectivity of the map \cite[p. 32]{scholtes2012introduction}. Since the piecewise linearity of $\varphi$ originates in absolute values that are not encapsulated in other absolute values, it is a simply switched piecewise linear function in the sense of \cite[Parag. 2]{griewank2014abs}. 

Also, by continuity of the determinant, there exists, for each signature $\Sigma$, an open neighborhood $M_{\Sigma}\subset \operatorname{M}_n(\mathbb R)$ about $I$ s.t. for all $\bar I\in M_{\Sigma}$ we have $\operatorname{det}(\bar I- S\Sigma)>0$. Then (the finite intersection of open sets) $M\equiv \bigcap_{\Sigma \in \operatorname{diag}_{n,\sigma}} M_{\Sigma}$ is a nonempty open neighborhood about $I$ s.t. $\operatorname{det}(\bar I- S\Sigma)>0$ for all $\bar I\in M$ and all $\Sigma\in \operatorname{diag}_{n,\sigma}$. Hence, the coherent orientation of $\varphi$ is stable under small perturbations of $I$. Thus, it conforms to the definition of a stably coherently oriented and simply switched piecewise linear map in \cite[Parag. 4.]{griewank2014abs}. And as such it is also injective (see \cite[Cor. 4.5.]{griewank2014abs}), hence bijective. 

$4.\Rightarrow 3.:$ Obvious. 
 \end{proof}

\begin{remark}
The equivalency $1.\Leftrightarrow 3.\Leftrightarrow 6. \Leftrightarrow 7.$ in Theorem \ref{unique}. was first stated by J. Rohn. 
The new proofs (mostly) use linear complementarity theory and thus showcase the kinship of LCPs and AVEs. 
For the linear algebraic original proofs, see, e.g., \cite[p. 218]{neumaier1990interval}.  
In \cite[Parag. 7.5]{griewank2014abs} Griewank proposed the transformation of a general PLE in so-called \textit{abs-normal} representation into an LCP. To prove $2.\Leftrightarrow 3.$ we adapted this reformulation to the AVE. 
The proof of $6.\Rightarrow 4.$ demonstrates the productive capacity of recent piecewise linear theory. 
The fact that the linear transformation $(I-S\Sigma)^{-1}$ maps $\hat c$ to a different orthant than the one defined by $\Sigma$ for all but one $\Sigma$ in $\operatorname{diag}_{n,\sigma}$ -- that is, point $5.$ -- is not interesting in the present setting, but gains significance in the context of Newton type approaches to the solution of \eqref{AVE}, such as those presented in \cite{griewank2014abs}.  
\end{remark}

Note that, by the SHIFT argument in the proof, the above statements still hold if we replace $I$ by $\alpha\cdot I$ and $ \rho_0^s(S)<1$ by $\rho_0^s(S)<\alpha$, respectively ($\alpha$ a positive scalar). 

Furthermore, if we keep in mind that multiplication by a signature matrix merely flips the signs of a row or column without changing the absolute values of the entries, we immediately see: 
\begin{align}\label{inftyEqual}
 \lVert \Sigma_1 S\lVert_{\infty} =  \lVert S \Sigma_2\lVert_{\infty} = \lVert \Sigma_1 S \Sigma_2 \lVert_{\infty}=\lVert S \lVert_{\infty}\ \forall\ \Sigma_1, \Sigma_2 \in \operatorname{diag}_{n,\sigma}.
\end{align}
Consequently, we also get:
\begin{align}\label{inftyBound}
\rho_0(\Sigma S)\le\rho (\Sigma S)\le \lVert \Sigma S\lVert_{\infty}=\lVert S \lVert_{\infty}\ \forall\ \Sigma \in \operatorname{diag}_{n,\sigma},
\end{align}
which implies $\rho_0^s(S) \le \lVert S\lVert_{\infty}$. As we we will only consider $S$ with $\lVert S\lVert_{\infty}<1$ in the present work, this yields $\rho_0^s(S)<1$ for all systems investigated hereafter and thus positively answers the question of their unique solvability. 

While we only make use of the infinity-case, it is worth mentioning that $\rho_0^s(S)$ is, in fact, bounded by all $p$-norms (see \cite[Thm. 2.15]{rump1997theorems}). Moreover, note that by the Perron-Frobenius rescaling introduced in \cite[Lem. 6.4]{griewank2014abs} any system $(I-S\Sigma)z=\hat c$ with $\Vert S\Vert_1<1$ can be transformed into a system $(I-S'\Sigma) z'= c'$ with $\Vert S'\Vert_\infty<1$.

\section{Main theorem}

We continue to use $S\in\operatorname M_n(\mathbb R)$ and $z,\hat c\in \mathbb R^n$ in their roles of the previous sections, and introduce a slight abuse of notation: Hereafter we will identify a vector $v\in \mathbb R^n$ with the set of its entries, ordered by their index. That is, $v\equiv \{v_1, \dots, v_n\}$. This way, the sets $C_{\max}$ and $ \Sigma_{\ne}$ in the definitions below can contain arbitrary subsets of the entries of $\hat c$. Now let
\begin{align*}
 C_{\max}\ \equiv\ \{ \hat c_j\in \hat c : \vert \hat c_j\vert = \max_{k\in [n]}(\vert \hat c_k\vert)\},
\end{align*}
and 
\begin{align*}
 \Sigma_{\ne}\ \equiv\ \{ \hat c_j\in \hat c : \sign(z_j)\ne \sign(\hat c_j)\},
\end{align*}
where $\sign$ denotes the \textit{signum function}. That is, $\sign$ is an element in \{-1, 0, 1\}. This is a stricter notion of sign coincidence than the one given in the introduction, where $0$ was essentially treated as a logical \textit{don't-care}, for which both $+$ and $-$ were allowed as proper signs. 


\begin{theorem}\label{mainThm}
Let $S\in\operatorname M_n(\mathbb R)$ and $z,\hat c\in \mathbb R^n$ such that it holds $$(I-S\Sigma)z\ =\ \hat c\ ,$$ where $\Sigma z=\vert z\vert$. Then we have $$C_{\max}\cap \Sigma_{\ne}\ =\  \emptyset\ ,$$ if one of the following conditions is satisfied:
 \begin{enumerate}
  \item $\lVert S\lVert_{\infty}\ <\ \frac 12$.
  \item $S$ is irreducible with $\lVert S\lVert_{\infty}\ \le\ \frac 12$.
  \item $S$ is strictly diagonally dominant with $\lVert S \lVert_{\infty}\ \le\ \frac 23$.
  \item $S$ is tridiagonal with $\lVert S \lVert_{\infty}\ <\ 1$.
 \end{enumerate}
\end{theorem}


\subsection{Proof of 1. and 2.}
The following lemma will provide a sufficient condition for the statement of the theorem to hold.

\begin{lemma}\label{sufficient}
 Let $\hat c, z \in \mathbb R^n$, $S\in \operatorname{M}_n(\mathbb R)$ with $\rho_0^s(S)<1$, and $\Sigma \in \operatorname{diag}_{n,\sigma}$, such that they satisfy \eqref{AVE}. Then, if the matrix $A\equiv (I-S\Sigma)^{-1}$ is strictly diagonally dominant with a positive diagonal, we have
 $$C_{\max}\cap \Sigma_{\ne}\ =\ \emptyset.$$
\end{lemma}
\begin{proof} Fix any $\hat c_i \in C_{\max}$. We distinguish two cases:

 \textit{Case 1:} Let $\vert \hat c_i\vert > 0$. Since $\vert \hat c_i\vert \ge \vert \hat c_j\vert\ \text{for all}\ j\in [n]$, we always have $\vert \hat c_i a_{ii}\vert > \sum_{j\ne i}\vert\hat c_j a_{ij}\vert$ due to the strict diagonal dominance of $A$. Consequently, since $a_{ii}$ is positive, $z_i$ will adopt the sign of $\hat c_i$. 
 
 \textit{Case 2:} Let $\hat c_i = 0$. Then $\hat c =\mathbf 0$, as $\hat c_i\in C_{\max}$. Hence, because of the unique solvability implied by $\rho_0^s(S)\le \vert S\vert_\infty <1$, $z$ is the zero vector as well -- which especially means $z_i=0$.
 \end{proof}

With this criterium in hand, we can prove the first two statements of the theorem:  

\begin{lemma}\label{diagDom}
 Let $S\in \operatorname{M}_n(\mathbb R)$ be irreducible with $\|S\|_\infty\le\frac12$,
then the inverse of the matrix $A\equiv I-S$ is strictly diagonally dominant and has a positive diagonal.
\end{lemma}
\begin{proof}
 We have $\|S^k\|_\infty\le\|S\|^k_\infty\le\frac1{2^k}$, which implies
$\lim_{k\to\infty}(I-A)^k=\lim_{k\to\infty}S^k=0$. Thus, $A^{-1}$ can be expressed
via the Neumann series
$$
A^{-1}=\sum_{k=0}^\infty(I-A)^k=\sum_{k=0}^\infty S^k = I+\sum_{k=1}^\infty
S^k.
$$
The inequality $\|\sum_{k=1}^\infty S^k\|_\infty\le
\sum_{k=1}^\infty\|S\|^k_\infty\le\sum_{k=1}^\infty\frac1{2^k}=1$
already ensures weak diagonal dominance of $A^{-1}$.

Now fix any $i\in [n]$ and assume that the $i$-th row
of $A^{-1}$ were not strictly dominated by its diagonal entry.
Denote the entries of $S^k$ by $s_{ij}^{(k)}$ for $i,j\in[n]$.
Then
$$
1=\sum_{j=1}^n\left|\sum_{k=1}^\infty s_{ij}^{(k)}\right|
\le\sum_{j=1}^n\sum_{k=1}^\infty|s_{ij}^{(k)}|\le
\sum_{k=1}^\infty\frac1{2^k}=1,
$$
which implies that
\begin{align}\label{cond1}
 \left|\sum_{k=1}^\infty s_{ij}^{(k)}\right|=\sum_{k=1}^\infty
|s_{ij}^{(k)}|\quad\forall j\in[n]
\end{align}
and $\sum_{j=1}^n |s_{ij}^{(k)}|=\frac1{2^k}$ for all $k\ge1$.
In particular,
\begin{align*}
\frac1{2^{k+1}}  &=\sum_{j=1}^n|s_{ij}^{(k+1)}|=\sum_{j=1}^n\left|\sum_{i=1}^n s_{ir}^{(k)}s_{rj}^{(1)}\right|\le \sum_{j=1}^n\sum_{r=1}^n|s_{ir}^{(k)}s_{rj}^{(1)}|\\
 &=\sum_{r=1}^n|s_{ir}^{(k)}|\sum_{j=1}^n|s_{rj}^{(1)}|\le \sum_{r=1}^n|s_{ir}^{(k)}|\cdot\frac12=\frac1{2^{k+1}},
 \end{align*}
which implies for each $k\ge1$ that
\begin{align}\label{cond2}
\left|\sum_{r=1}^n s_{ir}^{(k)}s_{rj}^{(1)}\right|=\sum_{r=1}^n|
s_{ir}^{(k)}s_{rj}^{(1)}|\quad\forall j\in[n].
\end{align}
\it Claim: \rm For each $k\ge1$, the $i$-th row of $S^k$ has the same
entry pattern as the $i$-th row of $|S|^k$.
\smallskip
We prove this by induction. The case $k=1$ is trivial. Assume the
claim holds for a given $k$. Let $\mathcal I_i^{(k)}=\{a_1,\ldots,a_m\}$
be the set of indices of the nonzero entries of the $i$-th row~$S^k$,
or equivalently of $|S|^k$. Define $\mathcal I_{a_1}^{(1)},\ldots,
\mathcal I_{a_m}^{(1)}$ analogously, and let $\mathcal I\equiv\bigcup_{a\in\mathcal I_i^{(k)}}
\mathcal I_a^{(1)}$. Obviously, $\mathcal I$ is precisely the set of indices
of the nonzero entries in the $i$-th row of $|S|^{k+1}=|S|^k|S|$,
and $s_{ij}^{(k+1)}\ne0$ at most if $j\in \mathcal I$. But this necessary
condition is also sufficient because otherwise \eqref{cond2}
would be violated. This completes the proof of the claim.

Since $|S|$ is irreducible and nonnegative, there exists a power
$|S|^{k_i}$ with a positive entry at $(i,i)$ (see, e.g., \cite[p. 3]{kitchens1998markov}).
By what we just showed, this implies $s_{ii}^{(k_i)}\ne0$. Therefore
$(s_{ii}^{(k_i)})^2>0$ and by~\eqref{cond2} also $s_{ii}^{(2k_i)}>0$.
Now \eqref{cond1} implies that $s_{ii}^{(k)}\ge0$ for all $k\ge1$.

Let $D$ be the diagonal part of $\sum_{k=1}^\infty S^k$ and
$B\equiv\sum_{k=1}^\infty S^k-D$.
Then $A^{-1}=I+D+B$, where $(I+D)_{ii}\ge 1+s_{ii}^{(2k_i)}>1$,
while $\sum_{j=1}^n |B_{ij}|\le 1-s_{ii}^{(2k_i)}<1$.
So our assumption that the $i$-th row of $A^{-1}$ were not
strictly dominated by its diagonal entry is in fact wrong. This completes
the proof.
\end{proof}

Note that for $\lVert S\lVert_1\leq \frac 12$ the arguments of the proof imply strict diagonal dominance of the inverse over the columns. Obviously, we also have strict diagonal dominance of $(\alpha A)^{-1}=\frac 1\alpha (A^{-1})$, where $\alpha \in\mathbb R \backslash \{0\}$. 

\begin{corollary}
Let $S\in \operatorname{M}_n(\mathbb R)$ with $\|S\|_{\infty} <\frac 12$,
then the inverse of the matrix $A\equiv I-S$ is strictly diagonally dominant and has a positive diagonal.
\end{corollary}
\begin{proof}
Consider the Neumann series $\sum_{k=0}^{\infty} S^k= I + \sum_{k=1}^{\infty} S^k$ in the proof above.  With the sharper bound we get $\|\sum_{k=1}^\infty S^k\|_\infty\le
\sum_{k=1}^\infty\|S\|^k_\infty<\sum_{k=1}^\infty\frac1{2^k}=1$. This ensures that $A^{-1}$ is strictly diagonally dominant with a positive diagonal.
\end{proof}

Now recall that, by \eqref{inftyEqual}, we have $\lVert S\Sigma\lVert_{\infty} = \lVert S\lVert_{\infty}$ for all $\Sigma \in\operatorname{diag}_{n,\sigma}$. 
Then it is clear that the restrictions stated in $1.$ and $2.$ in Theorem \ref{mainThm}. imply the strict diagonal dominance of $(I-S\Sigma)^{-1}$ for all $\Sigma\in\operatorname{diag}_{n,\sigma}$ -- which also includes the proper signature of the solution (in the sense that $\Sigma z =\vert z\vert$) and thus allows for the application of Lemma \ref{sufficient}. to the situation of the first two conditions. This completes the proof of the first two statements of Theorem \ref{mainThm}. 

\begin{remark} The matrix  
 $$S\equiv \begin{bmatrix} 0 & \frac 12 \\ 0 & \frac 12 \end{bmatrix}\qquad \text{with}\qquad (I-S)^{-1}=\begin{bmatrix} 1 & 1 \\ 0 & 2 \end{bmatrix}$$
 shows that in the limiting case $\lVert S \lVert_{\infty} =\frac 12$ the criterium of irreducibility cannot be omitted  in Lemma \ref{diagDom}. Furthermore, for $\epsilon >0$ arbitrarily small 
 $$S\equiv \begin{bmatrix} \epsilon & \frac 12 \\ 0 & \frac 12 \end{bmatrix}\qquad \text{with}\qquad (I-S)^{-1}= \begin{bmatrix} 1-2\epsilon & 1+2\epsilon \\ 0 & 2-2\epsilon \end{bmatrix}$$
 and $\lVert S\lVert_{\infty}=\frac 12 +\epsilon$ proves the sharpness of the bound $\lVert S\lVert_{\infty}\le\frac 12$.
 
 Also note that, if $S\in \operatorname M_n(\mathbb R)$ is nilpotent (which implies the nilpotency of $S\Sigma$ for all $\Sigma\in\operatorname{diag}_{n,\sigma}$), the Neumann expansion of $(I-S\Sigma)^{-1}$ has at most $n$ summands. Thus, if $\lVert S\lVert_{\infty}\le \frac 12$, we have 
 $$\left\|\sum_{k=1}^\infty S^k\right\|_\infty\ \le\ \sum_{k=1}^\infty\|S\|^k_\infty\ =\ \sum_{k=1}^{n-1}\|S\|^k_\infty\ \le\ \sum_{k=1}^{n-1}\frac1{2^k}\ <\ 1$$ -- 
 and again obtain strict diagonal dominance of the inverse of $I-S\Sigma$. But, since nilpotent matrices are permutationally similar to strictly upper triangular matrices, the corresponding AVEs can be solved by a modified backwards substitution in $\mathcal O(n^2)$ operations. Which is why we did not include this case in the main theorem.

\end{remark}

 \subsection{Proof of 3.}
Denote by $\operatorname{diag}_n(a_1, \dots, a_n)$ the $n$-dimensional diagonal matrix with entries $a_1, \dots, a_n\in \mathbb R$. Furthermore, define analogously to $C_{\max}$:
\begin{align*}
  Z_{\max}&\equiv \{  z_j\in z: \vert z_j\vert = \max_{k\in [n]}(\vert  z_k\vert)\}\ ,\\ 
  Z_{\min}&\equiv \{  z_j\in z: \vert z_j\vert = \min_{k\in [n]}(\vert  z_k\vert)\}\ .
 \end{align*}
We first exclude two special cases:
\begin{itemize}
\item  As $\rho_0^s(S)\le \lVert S\lVert_{\infty} <1$, the system  is uniquely solvable and the statement thus holds trivially for $z=\mathbf 0$. We therefore limit our attention to cases, where $z$ has at least one nonzero entry. 
  
\item $Z_{\min}$ and $Z_{\max}$ are either disjoint or equal. In both cases neither set is empty. Since $ \lVert S\lVert_{\infty} <1$, it is $\sign(z_i)=\sign(\hat c_i)$ for all $i\in[n]$, if $\vert z_1\vert=\dots =\vert z_n\vert$, i.e., if $Z_{max}=Z_{\min}$. Thus, we only have to prove $3.$ for cases, where $Z_{\max}\ne Z_{\min}$ and hence both sets are disjoint. 
\end{itemize}
The following observation is crucial:
\begin{itemize}
\item If $\lVert S\lVert_{\infty} <1$ and $z_i \in Z_{\max}$, we have $\sum_j \vert s_{ij} z_j\vert < \vert z_i\vert$ and hence $\mathbf{sign}(\hat c_i) = \mathbf{sign}(z_i)$. Consequently, if there were a tuple $(S,z,\hat c)$ that violated the claim of the theorem, for any $\hat c_j \in C_{\max}\cap \Sigma_{\ne}$ we would have $z_j\not\in Z_{\max}$. 
\end{itemize}
The proof is performed by induction. For $n=1$ the statement holds trivially. Assume it holds for $N\ge 1$, but there exists a tuple $(S,z, \hat c)$ in dimension $N+1$ that falsifies it. We distinguish two cases:

\textit{Case 1:} Let $\hat c_i \in C_{\max}$ and $z_i \not\in Z_{\min}$ s.t. $\mathbf{sign}(\hat c_i) \ne \mathbf{sign}(z_i)$. We will, from the falsifying tuple $(S,z, \hat c)$ in dimension $N+1$, construct a tuple $(\bar S,\bar z,\bar c)$ in dimension $N$ that falsifies the statement as well and thus contradicts the induction hypothesis. 
 
 Assume $w.l.o.g.$ that $z_{N+1}\in Z_{\min}$. Then for all $j\in [N]$ there exists a scalar $\zeta_j \in \left[0, 1 \right]$ such that $$\zeta_j \cdot \vert z_j\vert = \vert z_{N+1}\vert\quad \Longrightarrow\quad \zeta_j\cdot s_{j,N+1} \cdot \vert z_j\vert = s_{j,N+1}\cdot \vert z_{N+1}\vert\ .$$ 
 Denote by $S_{N+1,N+1}$ an $N$-dimensional square matrix derived from $S$ by removing row and column $N+1$. Then we have, for  
 \begin{align}\label{defZ}
  \bar z\ =\ (z_1,\dots, z_N)^T
 \end{align}
 and
 \begin{align}
   \bar S\ \equiv\ S_{N+1,N+1}\ +\ \operatorname{diag}_N(\zeta_1\cdot s_{j,N+1},\dots, \zeta_N\cdot s_{N,N+1})
 \end{align}
 that
 \begin{align}\label{defC}
 \bar z + \bar S \vert \bar z\vert\ =\ (\hat c_1,\dots, \hat c_N)^T \equiv\ \bar\ c\ .
 \end{align}
 Since the coefficients $\zeta_i$ are in $[0,1]$ for all $i\in[n]$, we have $\lVert \bar S\lVert_{\infty}\le \lVert S\lVert_{\infty}\le \frac 23$. For the same reason $\bar S$ is also still strictly diagonally dominant. Now, since $z_i \not\in Z_{\min}$, but $z_{N+1} \in Z_{\min}$, we must have $1\le i\le N$. That is, row $i$ (that contains the contradiction) was not removed by the construction. Thus, the tuple $(\bar S,\bar z, \bar c)$ contradicts the induction hypothesis for dimension $N$.
 
 \textit{Case 2:} Let $\hat c_i \in C_{\max}$ and $z_i \in Z_{\min}$ s.t. $$\mathbf{sign}(\hat c_i)\ \ne\ \mathbf{sign}(z_i)\ .$$ There is the possibility that $z_i$ is the only element in $Z_{\min}$. In this case the construction devised above fails, as it eliminates the row that contains the contradiction. We thus use an approach by direct computation. For this we note that, since $\lVert S\lVert_{\infty}<1$ and thus $s_{ii}<1$ for all $i\in [n]$, the following two statements hold:
 \begin{align*}
 \mathbf{sign}(z_i-s_{ii}\vert z_i\vert)\ =\ \mathbf{sign}(z_i) 
 \end{align*}
 and
  \begin{align}\label{tools2}
 \left|  z_i-s_{ii}\vert z_i\vert \right| \ \ge\ (1-\vert s_{ii}\vert)\vert z_i\vert\ . 
 \end{align}
 With $\hat c_i=z_i-\sum_{j=1}^{N+1}s_{ij}\vert z_j\vert$, and since $\mathbf{sign}(\hat c_i) \ne \mathbf{sign}(z_i)= \mathbf{sign}(z_i-s_{ii}\vert z_i\vert)$, it holds 
 \begin{align*}
\mathbf{sign}(z_i-s_{ii}\vert z_i\vert)\ \ne\ \mathbf{sign}\left(- \sum_{j\ne i}s_{ij}\vert z_j\vert\right)\ 
\end{align*}
and thus 
\begin{align*}
\vert \hat c_i\vert \ = \ \left\vert 
\vert z_i-s_{ii}\vert  z_i\vert\vert - \left\vert \sum_{j\ne i}s_{ij}\vert z_j\vert \right\vert 
\right\vert\ .
\end{align*}
 Using \eqref{tools2} then yields:
  \begin{align}\label{ineqMain1}
 \vert \hat c_i\vert\ & =\ \left| z_i - s_{ii}\vert z_i\vert - \sum_{j\ne i}s_{ij}\vert z_j\vert \right| \\ \label{ineqMain2}
 & \le\ \left| (1-\vert s_{ii}\vert)\vert z_i\vert - \sum_{j\ne i} \vert s_{ij}z_j\vert \right|\ \le\ \sum_{j\ne i} \vert s_{ij}z_j\vert\ . 
 \end{align}
 Furthermore, from $\sum_j \vert s_{ij}\vert \le \frac 23$ (norm constraint) and $\sum_{j\ne i} \vert s_{ij}\vert < s_{ii}$ (strict diagonal dominance), we get $\sum_{j\ne i} \vert s_{ij}\vert < \frac 13$. Now let $z_m\in Z_{\max}$. With \eqref{ineqMain1} and \eqref{ineqMain2} we get the leftmost inequality in:
 \begin{align}\label{contradictArg}
  \vert \hat c_i\vert\ \le\ \sum_{j\ne i} \vert s_{ij}z_j\vert\ \le\  \sum_{j\ne i} \vert s_{ij}z_m \vert\ =\ \vert z_m\vert \cdot \sum_{j\ne i} \vert s_{ij}\vert\ <\ \frac 13 \vert z_m\vert\ .
 \end{align}
 But we also have:
 $$\vert \hat c_m\vert\ =\ \left| z_m -\sum_j s_{mj}z_j\right|\ \ge\ \left| \vert z_m\vert - \sum_j\vert s_{mj}z_m\vert\right|\ \ge\ \frac 13\vert z_m\vert\ .$$
 Together with \eqref{contradictArg} the latter inequality gives $\vert \hat c_m\vert > \vert \hat c_i\vert$ -- which contradicts $\hat c_i\in C_{\max}$ and completes the proof of Theorem \ref{mainThm}.3.

\subsection{Proof of 4.}

The proof is again inductive. The case $n=2$ follows from a straightforward elementary calculation for which we refer to the appendix of \cite{radons2016master}. Now assume the statement of the theorem would hold for an $N\ge 2$, but the tuple $(S,z,\hat c)$ would contradict it in dimension $N+1$. We duplicate the argument from \textit{Case 1} in the proof of \textit{3.}: 
 
 As $N+1 \ge 3$, we can organize the system $w.l.o.g.$ such that $z_{N}\in Z_{\max}$ and $\hat c_{N+1}\not\in C_{\max}$. Then there exists a scalar $\zeta \in \left[0, 1 \right]$ such that 
 
 $$
 \zeta \cdot \vert z_N\vert = \vert z_{N+1}\vert\quad \Longrightarrow\quad \zeta\cdot s_{j,N+1} \cdot \vert z_N\vert = s_{j,N+1}\cdot \vert z_{N+1}\vert.
 $$ 
 Then $$\bar S\ \equiv\ S_{N+1,N+1} + \operatorname{diag}_N(0,\dots, 0, \zeta s_{N,N+1})$$ is still symmetric and tridiagonal with $\lVert S\lVert_{\infty}<1$. And, for $\bar z, \bar c$  defined as in \eqref{defZ} and \eqref{defC}, respectively, we have $\bar z+\bar S\vert \bar z\vert =\bar c$. Thus, the tuple ($\bar S,\bar z, \bar c$) contradicts the induction hypothesis for dimension $N$. 
 
 This completes the proof of the fourth statement and thus of the main theorem.

\section{Signed Gaussian elimination}

\subsection{Preliminaries}
We will show the three bullet points stated in the introduction:
Let $S$ and $\Sigma$ as in \eqref{AVE} and define the following matrix-block partitions: 
\begin{align}\label{blocks}
 \Sigma \equiv \begin{bmatrix} \sigma_1 & 0 \\ 0 & \bar\Sigma \end{bmatrix} \qquad \text{and} \qquad S \equiv \begin{bmatrix} E & F \\ G & H \end{bmatrix},
\end{align}
where $\sigma_1\in \{+1,-1\}$ is the first diagonal entry -- i.e., the first sign -- of $\Sigma$ and $E\equiv s_{11}$. Then the first step of a Gaussian elimination will transform $(I-S\Sigma)$ into

\begin{align*}
 &\begin{bmatrix} 1-\sigma_1s_{11} & -F\bar\Sigma \\ 0 & I-H\bar\Sigma +\sigma_1G(1-\sigma_1s_{11})^{-1}F\bar\Sigma \end{bmatrix}\\
 =&\begin{bmatrix} 1-\sigma_1s_{11} & -F\bar\Sigma \\ 0 & I-\bar S \bar\Sigma \end{bmatrix}, 
\end{align*}
where $\bar S\equiv H -\sigma_1G(1-\sigma_1s_{11})^{-1}F$.

As $\bar\Sigma$ is factored out, all one needs to calculate $\bar S$ and $1-\sigma_1s_{11}$ and thus be able to perform the first elimination step on the system matrix, is to choose a value for $\sigma_1$. 

Moreover, if we denote by $c'$ the updated vector $\hat c$ after one step of Gaussian elimination and define $\bar c\equiv (c'_2, \dots, c'_n)^T\in \mathbb R^{n-1}$, then it is
\begin{align}\label{gaussStep_c}
\bar c_i \equiv \hat c_{i+1} -\frac{\sigma_1\cdot (-s_{i+1,1})}{1-\sigma_1\cdot -s_{11}}\cdot \hat c_1= \hat c_{i+1} +\frac{\sigma_1\cdot s_{i+1,1}}{1-\sigma_1\cdot -s_{11}}\cdot \hat c_1
\end{align}
for all $i\in [n-1]$. And again this transformation can be performed, once $\sigma_1$ is fixed. Hence, one step of Gaussian elimination can be performed on the system \eqref{AVE} if $\sigma_1$ is fixed. 

 For the proper value of $\sigma_1$ we get a correct step. Here by correct we mean that the unique solution $\bar z$ of the reduced system equals the vector $(z_2,\dots ,z_n)^T$, that is, the elimination step is correct if the solution of the reduced system is identical to the last $n-1$ components of the solution $z$ of $(I-S\Sigma)z=\hat c$. One could, of course, also perform an elimination step with the wrong sign-choice. But then the equality of $\bar z$ and $(z_2,\dots ,z_n)^T$ would get lost.
This shows:

\begin{lemma}
If the sign of $z_1$ is known, then one correct step of Gaussian elimination can be performed on system \eqref{AVE}.
\end{lemma}

Furthermore, for $\lVert S\lVert_{\infty} <1$, the diagonal of the matrix $M\equiv I-S\Sigma$ is strictly positive. Since neither row and column pivots, nor multiplication with a signature [recall \eqref{inftyEqual}] change the infinity norm of $S$, this shows:

\begin{lemma}\label{permute}
If $\lVert S\lVert_{\infty}<1$, no row or column pivot in $S$ leads to numerical instabilities in the performance of a Gaussian elimination step on \eqref{AVE}. 
\end{lemma}

Thus, we can always, \textit{by symmetric row and column pivoting} (which is also called \textit{full pivoting} in some sources), produce a constellation for the AVE, where $\hat c_1\in C_{\max}$. Then Theorem \ref{mainThm}. provides us with the knowledge of the correct $\sigma_1$, if $S$ conforms to any of the conditions listed in the main result. If we want to perform more than only the first step of a Gaussian elimination applying this principle, we need the constraints to hold for the reduced subsystem(s) as well. The following technical lemma ensures this for all structural restrictions stated in the main theorem. 

\begin{lemma}\label{recursion}
 Let $S \in \operatorname{M}_n(\mathbb R)$ with $\lVert S \lVert_{\infty} = \xi < 1$, and define $\bar S$, $\bar\Sigma$ and $E,F,G,H$ as in \eqref{blocks}. Then the following statements hold:
 \begin{enumerate}
 \item $\lVert \bar S \lVert_{\infty} \le \xi < 1$.
 \item If $S$ is strictly diagonally dominant, then so is $\bar S$. 
 \item If $S$ is symmetric, then so is $\bar S$. 
 \item If $S$ is tridiagonal, then so is $\bar S$.

 \end{enumerate}
 \end{lemma}
 The proofs can be looked up in the appendix of \cite{radons2016master}. Note that \textit{4.} holds for arbitrary bandwidths of $S$.

\begin{remark}
It is not hard to find structural restrictions that allow for a loosening of the norm constraints on $S$, while $C_{\max}\cap \Sigma_{\ne}= \emptyset $ holds. The difficulty is that these restrictions have to be invariant under the reduction steps of a Gaussian elimination, which consist of a mere addition of an outer product to the subsystem. For example, antisymmetry will necessarily get lost, as the only antisymmetric outer product is the zero matrix.
\end{remark}
 
\subsection{The algorithm}
The key idea for the signed Gaussian elimination is simple: Pivot the entry in $C_{\max}$ with the smallest index (and the corresponding row and column) to the first position, assume its sign to be correct and set it as the $\sigma_1$ for the first elimination step. Then repeat the procedure for the reduced system and so forth.

Below is a pseudocode for the algorithm. It makes use of the following conventions: 

\begin{itemize}
\item $S, z,\hat c$ and $\Sigma$ are defined as in \eqref{AVE}.
\item $P_{jk}$ denotes the permutation matrix that corresponds to a transposition of $j$ and $k$.
\item $i_{C_{\max}}^{jn}$ denotes the smallest index $i$ with $j\le i\le n$, where $\hat c_i \in C_{\max}$.
\item $GaussStep(A,b,j)$ is the signature of a function that performs the $j$-th step of a Gaussian elimination on a system, where $Ax=b$.
\end{itemize} 
\smallskip

\begin{algorithm}
\caption{ Signed Gaussian elimination}

\begin{algorithmic}[1]
 
\STATE $P=I$

\FOR {$j=1:n$}

\STATE $k=i_{C_{\max}}^{jn}$

\STATE $\sigma_k = 1$

\IF{$\hat c_k<0$}\STATE{$\sigma_k = -1$}\ENDIF

\STATE $S=P_{jk}S P_{jk}$

\STATE $\Sigma=P_{jk}\Sigma P_{jk}$

\STATE $\hat c = \hat c P_{jk}$

\STATE $P=PP_{jk}$

\STATE $(I-S\Sigma,\hat c)=GaussStep(I-S\Sigma,\hat c, j)$

\ENDFOR

\STATE $z=(I-S\Sigma)^{-1}\hat c$

\STATE $z=zP$

\RETURN $z$

\end{algorithmic}

\end{algorithm}

\subsection{Correctness}
With the results gathered so far, the proof of correctness for the conditions described in Theorem \ref{mainThm}. is little more than a formality:

\begin{proposition}\label{SGE}
The SGE computes the unique solution of \eqref{AVE} correctly, if $S$ conforms to any of the conditions described in Theorem \ref{mainThm}. 
 \end{proposition}

\begin{proof}
For all cases we have $\rho_0^s(S)\le \lVert S\lVert_{\infty}<1$, which guarantees the unique solvability of \eqref{AVE} and allows for unproblematic (symmetric) pivoting of rows and columns (Theorem \ref{unique}. and Lemma \ref{permute}.). Theorem \ref{mainThm}. guarantees the correctness of the first sign choice. Lemma \ref{recursion}. assures that the conditions of the theorem are also satisfied by the reduced system. Hence, the argument applies recursively. 

For the tridiagonal case we remark that for $n=1$ and $\lVert S\lVert_{\infty} <1$ we always have $\sign(z_1)=\sign(\hat c_1)$. Hence, the reduction step from a two- to a one-dimensional subsystem is unproblematic with regard to the correctness of the result. Even though for a square matrix of dimension one the notion of tridiagonality clearly makes no sense.
\end{proof}

The proposition shows that, in a way, systems which conform to the conditions of the main theorem behave like dented linear systems rather than fully fledged piecewise linear systems.

\begin{remark}
Let $S$ and $z$ be generated uniformly at random. Then the expected value of $S\vert z\vert$ is the zero vector. This means, even though the infinity norm of $S$ may be arbitrarily large, for $\hat c\equiv z- S\vert z\vert$ the sign of $\hat c_i$ is a maximum likelihood estimate for the sign of $z_i$ for all $i\in [n]$ ($n$ the dimension of the system). So, for the SGE any of the popular testing of algorithms beyond their proven correctness range with randomly generated systems would be a rather pointless exercise: Relevant problem dimensions begin in the thousands, where the law of large numbers makes a false estimate highly unlikely.
\end{remark} 

\subsection{Effect on runtime}\label{cost}
Throughout this analysis we will assume a uniform cost model, i.e., elementary arithmetic operations, as well as reading, writing and comparing a floating point number are all assumed to be in $\mathcal O(1)$. 
It is well known that, within this model, the Gaussian eliminations for dense and tridiagonal matrices have a complexity in $\mathcal O(n^3)$ and $\mathcal O(n)$, respectively. (See, e.g., \cite[p. 752]{corman2007alg}, and \cite[p. 769]{corman2007alg}.)
The SGE has three types of additional operations in comparison to a classical Gaussian elimination \textit{without} pivoting: 
\begin{enumerate}
\item Determining the entry in $C_{\max}$ with the smallest index before every elimination step.
\item Pivoting in $S$ and $\hat c$ before the elimination step.
\item Permuting the entries of the solution into their correct order after the completed backwards substitution.
\end{enumerate}


For dense matrices this means that the SGE has precisely the cost of a Gaussian elimination with symmetric row/column-pivoting, which is roughly $\frac 13n^3$ fused multiply-adds (see the above references). So the SGE for dense matrices has the same asymptotical complexity as the unaltered algorithm. (For a detailed account of the operations of the different types of Gaussian elimination, see, e.g., \cite[pp. 744-752]{corman2007alg}.)  

For tridiagonal systems the second and third point can clearly be handled in $\mathcal O(n)$. However, an analysis of the first point shows that the additional operations increase the asymptotical complexity of the tridiagonal SGE in comparison to the tridiagonal Gaussian elimination:



For simplicity we assume that every elimination step produces no zeros beyond the column that is eliminated. That is, the reduced subsystems stay densely tridiagonal in the sense that the three diagonals have no zero entries. Then in every column there is exactly one nonzero entry below the principal diagonal. Thus, the $i$-th elimination step exclusively affects row $(i+1)$ of $S$ and thus only entry $(i+1)$ of $\hat c$. That is to say: $\hat c_{i+2}$ to $\hat c_n$ remain unaltered. Accordingly, it would be inefficient to run a comparison of all remaining entries of $\hat c$ after each elimination step. We outline a better approach: 

Assume that $\hat c_{i+2}$ to $\hat c_n$ are sorted by absolute value (highest first) before the $i$-th elimination step. The only entry of $\hat c$ updated in the $i$-th step is $\hat c_{i+1}$. Then, to determine the entry between $\hat c_{i+1}$ and $\hat c_n$ with the largest absolute value, one only has to compare $\hat c_{i+1}$ and $\hat c_{i+2}$. The only entry of $\hat c$ updated in the next elimination step  (after swapping $\hat c_{i+1}$ and $\hat c_{i+2}$, if necessary) is $\hat c_{i+2}$. And $\hat c_{i+3}$ to $\hat c_n$ remain sorted by absolute value. Hence, the argument applies recursively. 

Now let $i=1$, i.e., sort $\hat c$ before the first elimination step. Then, afterhand we need only $n-1$ comparisons and at most $n-1$ swaps throughout the elimination, which is clearly in $\mathcal O(n)$. As sorting $n$ floats has the trivial lower bound $\mathcal O(n)$, but is currently not possible with  this efficiency, the overall complexity of determining the proper order for the elimination is bounded from below by the complexity of the utilized sorting algorithm. 

Asymptotically this approach is optimal, since determining the entry of $\hat c$ with the largest absolute value in every step clearly has sorting $\hat c$ once as a lower bound. 


Hence, the tridiagonal SGE is at least as expensive as the algorithm utilized to sort $\hat c$. Since the tridiagonal Gaussian elimination's complexity is in $\mathcal O(n)$, this especially means that the tridiagonal SGE has -- at the present state of research -- a higher asymptotical complexity than the unmodified algorithm. 
(Note that, apart from a higher constant factor, this result holds for any fixed bandwidth of $S$.)

Currently, the asymptotically fastest sorting algorithm for floating point numbers, developed by Han and Thorup in \cite{han2002sort}, has a complexity of $\mathcal O (n\cdot\sqrt{\operatorname{log}\operatorname{log} n})$.  
But this is only a  theoretical performance, since the latter is inefficient for realistic problem dimensions. For an actual application the use of an easily implementable in-place sorting algorithm such as Quicksort with its $\mathcal O(n\cdot \operatorname{log} n)$ average cost (see, e.g., \cite[pp. 143-161]{corman2007alg}) is a far more adequate choice.





\section{Sharpness of the bounds}
For $n=1$  we always have $\mathbf{sign}(z_1)=\mathbf{sign}(\hat c_1)$, if $\lVert S\lVert_{\infty}<1$. So, naturally, we are inclined to ask whether the bounds from Theorem \ref{mainThm}. can be loosened further. 

\begin{proposition}\label{complication}
Let $S\in \operatorname{M}_n(\mathbb R)$ with $\lVert S\lVert_{\infty}\le \frac 12$, and $z,\hat c\in \mathbb R$ s.t. $z+S\vert z\vert=\hat c$. Then for $n\ge 2$ the following holds:
\begin{enumerate}
\item It is possible that there exists a $c_i\in C_{\max}$ such that $\mathbf{sign}(\hat c_i)\ne \mathbf{sign}(z_i)$.
\item If $c_i\in C_{\max}$ and $\mathbf{sign}(\hat c_i)\ne \mathbf{sign}(z_i)$, then $z_i=0$.
\end{enumerate}
\end{proposition}  
\begin{proof}
\textit{1.:} Let 
 \begin{center}
 $S\equiv\begin{bmatrix}  0 & \frac 12 & 0 & . & 0\\ 0 & \frac 12 & 0 & . & 0 \\ . & . & . & . & .\\0 & \frac 12 & 0 & . & 0\\  0 & \frac 12 & 0 & . & 0\end{bmatrix} \qquad \text{and} \qquad    z\equiv\begin{bmatrix} 0  \\ 1 \\ \dots \\ 1 \\1 \end{bmatrix}$,
\end{center}
then $\hat c=(-\frac 12,\frac 12,\dots, \frac 12)^T$ and we clearly have $c_1 \in C_{\max}$, but $\mathbf{sign}(\hat c_1)\ne \mathbf{sign}(z_1)$.

\textit{2.:}
 We first exclude the following special case from the discussion: If we have $0=\hat c_i\in C_{\max}$, then $\hat c=\mathbf 0$ and thus, by unique solvability of the system, $z$ is the zero vector as well. That is, we cannot have $\mathbf{sign}(\hat c_i)\ne \mathbf{sign}(z_i)$.
 
 Now assume there were a $0\ne \hat c_i \in C_{\max}$ with $\sign(z_i) \ne \sign(\hat c_i)$ and $z_i\ne 0$. 
 Let $z_j\in Z_{\max}$. Since $\lVert S\lVert_{\infty} \le \frac 12$, we have:
  \begin{align}\label{smaller}
 \frac12 \cdot\vert z_j\vert \ge   \left| e_k^TS\vert z\vert \right| \qquad \forall k\in [n].
  \end{align} 
As $\hat c_j = z_j -e_j^TS\vert z\vert$, this especially gives $\vert\hat c_j\vert \ge \frac 12\cdot \vert z_j\vert$. We proceed by a case distinction:
 
\textit{Case 1:} Let $z_i\in Z_{\max}.$ If $z_i$ were zero, then $z$ would be the zero vector and thus $\hat c$ as well. If we had $\vert z_i\vert >0$, then $\hat c_i$ would have to adopt the sign of $z_i$ due to \eqref{smaller}. Hence, we would have $\sign(z_i) = \sign(\hat c_i)$ in contradiction of the initial assumption.

\textit{Case 2:} Let $z_i\not\in Z_{\max}$. Since $\sign(z_i) \ne \sign(\hat c_i)$, but $\hat c_i = z_i - e_i^TS\vert z\vert$, the sign of $\hat c_i$ must be the same as that of $- e_j^TS\vert z\vert$. This gives the leftmost inequality in 
\begin{align*}
\vert \hat c_i \vert \le  \left| e_i^TS\vert z\vert \right| \le \frac 12\cdot \vert z_j\vert\le \vert \hat c_j\vert,
\end{align*}
where $z_j\in Z_{\max}$. Since $\hat c_i\in C_{\max}$, we have $\vert \hat c_i\vert = \vert e_i^TS\vert z\vert\vert  =\vert \hat c_j\vert$ -- and the left equality clearly yields $z_i=0$. 
\end{proof}

The second statement of the proposition makes sure that the SGE calculates the proper solution for arbitrarily structured $S$ with $\lVert S\lVert_{\infty}\le \frac 12$, while the statement of the main theorem does not hold anymore in its strict sense that $z_i=0$ if and only if  $0=\hat c_i\in C_{\max}$. That is, the SGE also computes solutions on orthant boundaries correctly.

Also note that, by replacing the inequalities in \eqref{smaller} with strict inequalities, the proof of Proposition \ref{complication}.2. can be used as an alternate proof for Theorem \ref{mainThm}.1. 

One might ask now, if the SGE still runs provably correct with irreducible $S$ that have a norm greater than one half. We will see below that the answer to this query is no. Accordingly, under purely practical considerations the first two points of Theorem \ref{mainThm}. could be merged into one condition: $S$ arbitrarily structured with $\lVert S\lVert_{\infty}\le \frac 12$.

\begin{proposition}
 For an irreducible $S \in \operatorname{M}_n(\mathbb R)$ with $n\ge 2$, the correctness of the SGE cannot be ensured, if $\lVert S\lVert_{\infty} >\frac 12$.
\end{proposition}

\begin{proof}
 We start by demonstrating the sharpness of the bound for $n= 2$. For an $\epsilon >0$ let 
 
 $$S\equiv \begin{bmatrix} \frac{\epsilon}{2}& \frac{1+\epsilon}{2}  \\ 0 & \frac 12 \end{bmatrix} \qquad \text{and} \qquad    z\equiv\begin{bmatrix} \frac{\epsilon}{2}  \\ 1 \end{bmatrix}.$$
Then, for $\hat c\equiv z-S\vert z\vert$ we have $\hat c=(-\frac{2+\epsilon}{4}, \frac 12)^T$. And clearly $\vert c_1\vert >\vert c_2\vert$, but  $\mathbf{sign}(\hat c_1)\ne \mathbf{sign}(z_1)$. 

 The structure of this example can be extended to higher dimensions. Let
\begin{center}
 $S\equiv\begin{bmatrix}  \frac{\epsilon}{2} & \frac{1+\epsilon}{2} & 0 & . & 0\\ 0 & 0 & 0 & . & \frac 12 \\ . & . & . & . & .\\0 & 0 & \frac 12 & . & 0\\  0 & \frac 12 & 0 & . & 0\end{bmatrix} \qquad \text{and} \qquad    z\equiv\begin{bmatrix} \frac{\epsilon}{2}  \\ 1 \\ \dots \\ 1 \\1 \end{bmatrix}$.
\end{center}
This yields $\hat c=(-\frac{2+\epsilon}{4},\frac 12,\dots, \frac 12)^T$. And again: $c_1 \in C_{\max}$, but $\mathbf{sign}(\hat c_1)\ne \mathbf{sign}(z_1)$. 

As in both cases $S$ is irreducible with $\lVert S\lVert_{\infty} =\frac 12 +\epsilon$, this establishes the sharpness of the bound for $n \ge 2$. 
\end{proof}


In the tridiagonal case the bound is sharp:

\begin{proposition}
Let $S\in \operatorname{M}_n(\mathbb R)$ be tridiagonal and symmetric. If $\lVert S \lVert_{\infty} \ge 1$, then the correctness of the SGE cannot be ensured.    
\end{proposition}
\begin{proof}
 Just consider $S=-I$ and $z$ the vector with entries $-1$. Then $\hat c =z +\vert z\vert=\mathbf 0$ -- and the SGE fails, since it picks $+1$ as $\sigma_1$. 
\end{proof}
On a more general note, keep in mind that for $\lVert S\lVert_{\infty}\ge 1$ the unique solvability cannot be guaranteed anymore, which means that we enter an altogether different problem sphere.

We did not manage to find a counterexample that establishes the absolute sharpness of the bound in Theorem \ref{mainThm}.3. However, we can demonstrate that the norm constraint for strictly diagonally dominant matrices can be loosened at most by a minute amount:

\begin{proposition}
Let $S\in \operatorname{M}_n(\mathbb R)$ be strictly diagonally dominant. If $\lVert S \lVert_{\infty} \ge \frac 23+\frac{1}{3(n+1)}$, then the correctness of the SGE cannot be ensured.   
\end{proposition}
\begin{proof}
Let 
 \begin{center}
 $S\equiv\begin{bmatrix}  \frac 13 + \frac{1}{3(n+1)} & \frac{1}{3(n-1)} & \frac{1}{3(n-1)} & . & \frac{1}{3(n-1)}\\ 0 & \frac 23 +\frac{1}{3(n+1)}& 0 & . & 0 \\ 0 & 0 & \frac 23+\frac{1}{3(n+1)} & . & 0\\ . & . & . & . & .\\  0 & 0 & 0 & . & \frac 23+\frac{1}{3(n+1)}\end{bmatrix}\quad \text{and}\quad     z\equiv\begin{bmatrix} \frac{n+1}{2n+1}\cdot \epsilon  \\ 1 \\ 1 \\ \dots \\1 \end{bmatrix}$.
\end{center}
 Then $\hat c=(-\frac 13 \cdot (1-\epsilon), \frac 13 \cdot \frac{n}{n+1}, \dots, \frac 13 \cdot \frac{n}{n+1})^T$. Now choose an $\epsilon>0$ such that $1-\epsilon>\frac{n}{n+1}$. Then $\hat c_1$ has the largest absolute value of all entries in $\hat c$, but $\sign(z_1) \ne \sign(\hat c_1)$ -- even in the strict sense that $\hat c_i <0$, but $z_i>0$. Hence, the first sign choice of the SGE fails.
 \end{proof}
 
 \section{Acknowledgements}
The author is indebted to his teacher Andreas Griewank and his colleague Tom Streubel, upon whose research the findings presented above are built. He also wants to convey his gratitude to Dorothee Schueth, Gestur Olafsson, and Jens Uwe Bernt for the time they invested into discussing and proofreading this work. A special thanks goes to Pascal Lenzner and Katharina Klost for their help with the presentation of the algorithmic section.

 \bibliographystyle{alpha}
\bibliography{directSol}
\end{document}